\documentclass[12pt]{article}

\usepackage{amsmath}
\usepackage{amsthm}
\usepackage{amsfonts}
\usepackage{amscd}

\newcommand{\Z}{{\mathbb Z}}
\newcommand{\R}{{\mathbb R}}
\newcommand{\D}{{\partial}}
\newcommand{\rD}{{\bar{\partial}}}
\newcommand{\genh}{{\mathfrak H}}
\newcommand{\rgenh}{\bar{\mathfrak H}}
\newcommand{\M}{{\mathfrak M}}
\newcommand{\A}{{\cal A}}

\DeclareMathOperator{\im}{im}

\DeclareMathOperator{\Int}{int}
\DeclareMathOperator{\sign}{sign}

\newtheorem{lemma}{Lemma}[section]
\newtheorem{axiom}{Axiom}

\newtheorem{theorem}[lemma]{Theorem}
\newtheorem{proposition}[lemma]{Proposition}
\newtheorem{definition}[lemma]{Definition}

\title{Generalized equivariant homology on simplicial complexes}
\author{jason hanson}
\date{}

\begin{document}

\maketitle

\begin{abstract}
A careful account is given of generalized equivariant homology theories on the category of topological pairs acted on by a group.  In particular, upon restriction to the category of equivariant simplicial complexes, the equivalence of equivariant simplicial homology (also known as Bredon homology), the second derived term of the Atiyah--Hirzebruch spectral sequence, and equivariant singular homology is demonstrated.
\end{abstract}

\section{Introduction and notation}

Throughout, we let $G$ denote a topological group.  On the category of pairs of $G$--complexes (cw--complexes with action by $G$), an extensive discussion of generalized equivariant cohomology theories is given in \cite{Bredon1}; and that for generalized equivariant homology theories, although not as extensive in scope, is given in \cite{Wilson}.  Our goal here is twofold.  First, we wish to give a discussion of equivariant homology theories for a larger category of $G$--spaces, and to bridge some of the gaps present in both of the above references.  Second, we wish to see how the theory plays out when restricted to the category of equivariant simplicial complexes; that is, we wish to provide a generalized equivariant version of the axiomatic approach given in \cite{ES}.

For the first half of this paper, we will work in the category of arbitrary $G$--pairs and equivariant maps.  Specifically, the objects of the category consist of pairs $(X,A)$, where the topological spaces $A\subseteq X$ have continuous (left) $G$--action; i.e., $A$ is a $G$--subspace of $X$.  The morphisms of the category are continuous equivariant maps: continuous maps $f:(X,A)\rightarrow(Y,B)$ between $G$--pairs such that $f(gx)=gf(x)$, for all $g\in G$ and $x\in X$.

In the second half, we will restrict our attention to simplicial complexes with $G$--action.  For us, these will be topological spaces that can be decomposed as a union of {\em equivariant cells} of the form $G/H\times\Delta_p$, where $\Delta_p$ is the standard $p$--simplex given the trivial $G$--action, and $H\leq G$ (that is, $H$ is a subgroup of $G$).

In an effort to reduce some of the notational clutter, we will introduce the following conventions.  First, we will use the same label for a map and its restriction when our intentions are unambiguous.  E.g., if $f:X\rightarrow Y$ is a map and $A\subseteq X$, $B\subseteq Y$ are such that $f(A)\subseteq B$, then we will write $f:A\rightarrow B$ for the corresponding restriction of domain and range.  Second, any unlabeled map should be assumed to be the inclusion map; e.g., $A\rightarrow X$.  Third and finally, the previous two conventions take precedence over any functor; for example, if $F$ is a covariant functor, then $F(f):F(A)\rightarrow F(B)$ and $F(A)\rightarrow F(X)$ denote the functor $F$ applied to the restriction $f:A\rightarrow B$ and the inclusion $A\rightarrow X$, respectively.

\section{Generalized equivariant homology}

\subsection{Axioms}

Let $R$ be a ring, and suppose that $\genh_\ast^G$ is a covariant functor from the category of $G$--pairs and equivariant maps to the category of graded $R$--modules and linear maps.  In particular if $(X,A)$ is a $G$--pair, then $\genh_\ast^G(X,A)=\oplus_{n\in\Z}\genh_n^G(X,A)$ is a graded $R$--module.  For a $G$--space $X$, we set $\genh_\ast^G(X)\doteq\genh_\ast^G(X,\emptyset)$.  Suppose that in addition, $\D_\ast$ is a natural transformation of functors such that for $G$--pairs $(X,A)$, $\D_n:\genh_n^G(X,A)\rightarrow\genh_{n-1}^G(A)$ is $R$--linear for all $n\in\Z$.  The pair $(\genh_\ast^G,\D_\ast)$ is said to be a {\bf generalized equivariant homology theory} if the following three axioms are satisfied.

\begin{axiom}[Exactness]\label{ax:exact}
If $(X,A)$ is a $G$--pair, then there is a long exact sequence of $R$--modules
$$\genh_n^G(A)
  \rightarrow
  \genh_n^G(X)
  \rightarrow
  \genh_n^G(X,A)
  \xrightarrow{\D_n}
  \genh_{n-1}^G(A).
$$
\end{axiom}

\begin{axiom}[Homotopy]\label{ax:homotopy}
Suppose that the maps $f,g:(X,A)\rightarrow(Y,B)$ are $G$--homotopic: there exists an equivariant map $F:(X,A)\times[0,1]\rightarrow(Y,B)$, where the interval $[0,1]$ is given the trivial $G$--action.  Then $\genh_\ast^G(f)=\genh_\ast^G(g)$.
\end{axiom}

\begin{axiom}[Excision]\label{ax:excision}
If $(X,A)$ is a $G$--pair and $U$ is a $G$--subspace such that $\bar{U}\subseteq\Int(A)$, then $\genh_\ast^G(X\setminus U,A\setminus U)\rightarrow\genh_\ast^G(X,A)$ is an isomorphism.
\end{axiom}

Note that as a consequence of the exactness axiom, $\genh_\ast^G(A,A)=0$; and in particular, $\genh_\ast^G(\emptyset)=\genh_\ast^G(\emptyset,\emptyset)=0$.

One may also impose conditions involving the behavior under group homomorphisms, see \cite{Lueck}; however, we will not have the need for these.

\subsection{Fundamental properties}

Many of the familiar properties of nonequivariant singular homology carry over to generalized equivariant homology.  We explicitly state a few for reference.

\begin{proposition}[Exact sequence of triple]
\label{prop:triple}
If $(X,B,A)$ is a triple of $G$--spaces, then there is a long exact sequence of $R$--modules:
$$\genh_n^G(B,A)
  \rightarrow
  \genh_n^G(X,A)
  \rightarrow
  \genh_n^G(X,B)
  \xrightarrow{\D_n}
  \genh_{n-1}^G(B,A),
$$
where $\D_\ast$ here is the composition $\genh_\ast^G(X,B)\xrightarrow{\D_\ast}\genh_{\ast-1}^G(B)\rightarrow\genh_{\ast-1}^G(B,A)$.
\end{proposition}

\begin{proposition}[Naturality of sequences]
\label{prop:long-nat}
Given an equivariant map of $G$--triples $f:(X,B,A)\rightarrow(Y,D,C)$, there is a homomorphism of long exact sequences (that is, each square commutes):
$$\begin{CD}
  \genh_n^G(B,A)
    @>>>
      \genh_n^G(X,A)
        @>>>
          \genh_n^G(X,B)
            @>{\D_n}>>
              \genh_{n-1}^G(B,A)\\
  @V{\genh_n^G(f)}VV
    @V{\genh_n^G(f)}VV
      @V{\genh_n^G(f)}VV
        @V{\genh_{n-1}^G(f)}VV\\
  \genh_n^G(D,C)
    @>>>
      \genh_n^G(Y,C)
        @>>>
          \genh_n^G(Y,D)
            @>{\D_n}>>
              \genh_{n-1}^G(D,C).
\end{CD}$$
\end{proposition}

\begin{proposition}[Strong excision]
\label{prop:excision}
Suppose that $(X,A)$ is a $G$--pair and $U$ an open $G$--subspace of $A$.  If there exists an open $G$--subspace $V$ with $\bar{V}\subseteq U$ and such that $(X\setminus U,A\setminus U)$ is a deformation retract of $(X\setminus V,A\setminus V)$, then $\genh_\ast^G(X\setminus U,A\setminus U)\rightarrow\genh_\ast^G(X,A)$ is an isomorphism.
\end{proposition}

\begin{proposition}[Finite disjoint unions]
\label{prop:findisj}
Suppose $(X_k,A_k)$ is a $G$--pair for each $1\leq k\leq n$, and set $X\doteq\sum_{k=1}^nX_k$, $A\doteq\sum_{k=1}^nA_k$.  We have an isomorphism $\Sigma_\ast:\oplus_{k=1}^n\genh_\ast^G(X_k,A_k)\rightarrow\genh_\ast^G(X,A)$, defined by requiring $\Sigma_\ast$ to be equal to $\genh_\ast^G(X_k,A_k)\rightarrow\genh_\ast^G(X,A)$ when restricted to summands $\genh_\ast^G(X_k,A_k)$.
\end{proposition}

Proofs of the above are exactly the same as for a non--generalized non--equivariant homology theory: see the proofs of theorems 10.2, 10.3, 15.3, 15.4, 12.1, and 13.2 (respectively) of chapter I in \cite{ES}.

\subsection{Compact supports and disjoint unions}

The analogous statement in proposition \ref{prop:findisj} for infinite disjoint unions requires additional assumptions.  One possible route is to follow \cite{ES} and impose a compact support axiom.

\begin{definition}
The generalized equivariant homology theory $(\genh_\ast^G,\D_\ast)$ has {\bf compact supports} if for any $G$--pair $(X,A)$ and any $x\in\genh_\ast^G(X,A)$, there exists a compact $G$--pair $(X_c,A_c)\subseteq(X,A)$ and $x_c\in\genh_\ast^G(X_c,A_c)$ such that $x$ is the image of $x_c$ under $\genh_\ast^G(X_c,A_c)\rightarrow\genh_\ast^G(X,A)$.
\end{definition}

However, the compact support assumption may be too much to ask for in some situations; for instance, when $G$ is not compact.  In this case, we gain a fair amount by simply imposing the arbitrary disjoint union property.

\begin{definition}
The generalized equivariant homology theory $(\genh_\ast^G,\D_\ast)$ satisfies the {\bf arbitrary disjoint union property} if for any indexing set $\A$ and any collection of $G$--pairs $\{(X_\alpha,A_\alpha)\}_{\alpha\in\A}$, the map $\Sigma_\ast:\oplus_{\alpha\in\A}\genh_\ast^G(X_\alpha,A_\alpha)\rightarrow\genh_\ast^G(X,A)$ is an isomorophism.  Here $X\doteq\sum_{\alpha\in\A}X_\alpha$ and $A\doteq\sum_{\alpha\in\A}A_\alpha$, and $\Sigma_\ast$ is uniquely defined by the requirement that it coincides with $\genh_\ast^G(X_\alpha,A_\alpha)\rightarrow\genh_\ast^G(X,A)$ on summands $\genh_\ast^G(X_\alpha,A_\alpha)$.
\end{definition}

\begin{proposition}[Direct limits]\label{prop:limit}
Let $(X,A)$ be a $G$--pair, and let ${\cal C}$ denote the directed set of all compact $G$--pairs contained in $(X,A)$.  If $(\genh_\ast^G,\D_\ast)$ has compact supports, then the inclusion maps $(Y,B)\rightarrow(X,A)$ induce an isomorphism $\lim_{(Y,B)\in{\cal C}}\genh_\ast^G(Y,B)\cong\genh_\ast^G(X,A)$.
\end{proposition}

The converse of this statement is also true; the proofs of both statements are the same as that of theorem 13 of chapter 4, section 8, in \cite{Spanier}.

\begin{theorem}[Arbitrary disjoint unions]\label{thm:arbdisj}
If a generalized equivariant homology theory has compact supports, then it satisfies the arbitrary disjoint union property.
\end{theorem}

\begin{proof}
For any compact pair $(Y,B)\subseteq(X,A)$, $Y=\sum_{\alpha\in\mathcal{A}}(Y\cap X_\alpha)$ and $A=\sum_{\alpha\in\mathcal{A}}(B\cap A_\alpha)$ are necessarily finite sums.  Thus in the commutative diagram:
$$\begin{CD}
  \oplus_{\alpha\in\mathcal{A}}
      \genh_\ast^G(Y\cap X_\alpha,B\cap A_\alpha)
    @>>>
      \genh_\ast^G(Y,B)\\
  @VVV
    @VVV\\
  \oplus_{\alpha\in\mathcal{A}}\genh_\ast^G(X_\alpha,A_\alpha)
    @>>>
      \genh_\ast^G(X,A),
\end{CD}$$
the top horizontal arrow is an isomorphism by proposition \ref{prop:findisj}.  The theorem now follows from propostition \ref{prop:limit} by taking direct limits over compact pairs (see theorem 4.13 of chapter VII in \cite{ES}).
\end{proof}

\section{Reduced homology}
\label{sec:redhom}

Analogous with the non--equivariant case, a {\em pointed $G$--space} is a pair $(X,x_0)$ where $X$ is a $G$--space and $x_0\in X$ is a distinguished point (the {\em base point}).  A {\em morphism} $f:(X,x_0)\rightarrow(Y,y_0)$ of pointed $G$--spaces is an equivariant map that preserves base points: $f(x_0)=y_0$.  We do not assume that the base point is fixed by the $G$--action; that is, a morphism $f$ of $G$--pairs induces an equivariant map of orbits $f:Gx_0\rightarrow Gy_0$, and hence an equivariant map of $G$--pairs $f:(X,Gx_0)\rightarrow(Y,Gy_0)$.

\begin{definition}\label{def:rhomology}
The {\bf reduced homology} of a pointed $G$--space $(X,x_0)$ is defined to be $\rgenh_\ast^G(X,x_0)\doteq\genh_\ast^G(X,Gx_0)$.  If $f:(X,x_0)\rightarrow(Y,y_0)$ is a morphism of pointed $G$--spaces, then $\rgenh_\ast^G(f):\rgenh_\ast^G(X,x_0)\rightarrow\rgenh_\ast^G(Y,y_0)$ is defined to be the homomorphism $\genh_\ast^G(f):\genh_\ast^G(X,Gx_0)\rightarrow\genh_\ast^G(Y,Gy_0)$.
\end{definition}

The following properties of reduced homology are immediate from the corresponding properties in non--reduced homology.

\begin{theorem}
\label{thm:rfunctor}
If $f:(X,x_0)\rightarrow(Y,y_0)$ and $g:(Y,y_0)\rightarrow(Z,z_0)$ are morphisms of pointed $G$--spaces, then $\rgenh_\ast^G(g\circ f)=\rgenh_\ast^G(g)\circ\rgenh_\ast^G(f)$.
\hfill\qed
\end{theorem}

\begin{theorem}
\label{thm:rexact}
If $(X,A)$ is a $G$--pair with $A\neq\emptyset$, then for any $a_0\in A$, there is a long exact sequence
$$\rgenh_n^G(A,a_0)
  \rightarrow
  \rgenh_n^G(X,a_0)
  \rightarrow
  \genh_n^G(X,A)
  \xrightarrow{\rD_n}
  \rgenh_{n-1}^G(A,a_0)
$$
(all maps induced by the corresponding maps in nonreduced homology).
\hfill\qed
\end{theorem}

\begin{theorem}
\label{thm:redhom-nat}
If $f:(X,A)\rightarrow(Y,B)$ is an equivariant map of $G$--pairs and $a_0\in A,b_0\in B$ such that $f(a_0)=b_0$, then $f$ induces a homomorphism of long exact sequences
$$\begin{CD}
  \rgenh_n^G(A,a_0)
    @>>>
      \rgenh_n^G(X,a_0)
        @>>>
          \genh_n^G(X,A)
            @>{\rD_n}>>
              \rgenh_{n-1}^G(A,a_0)\\
  @V{\rgenh_n^G(f)}VV
    @V{\rgenh_n^G(f)}VV
      @V{\genh_n^G(f)}VV
        @V{\rgenh_{n-1}^G(f)}VV\\
  \rgenh_n^G(B,b_0)
    @>>>
      \rgenh_n^G(Y,b_0)
        @>>>
          \genh_n^G(Y,B)
            @>{\rD_n}>>
              \rgenh_{n-1}^G(B,b_0).\hfill\qed
\end{CD}$$
\end{theorem}

A {\em $G$--homotopy} of pointed $G$--spaces $(X,x_0)$ and $(Y,y_0)$ is an equivariant map $F:(X,Gx_0)\times[0,1]\rightarrow(Y,Gy_0)$ of $G$--pairs (where $[0,1]$ is given the trivial $G$--action) such that $F(x_0,t)=y_0$ for all $t\in[0,1]$.

\begin{theorem}
\label{thm:rhomotopy}
If $f,g:(X,x_0)\rightarrow(Y,y_0)$ are $G$--homotopic maps of pointed $G$--spaces, then $\rgenh_\ast^G(f)=\rgenh_\ast^G(g)$.
\hfill\qed
\end{theorem}

In nonequivariant homology, one has the luxury of equating reduced homology with the kernel of the homology functor applied to the augmentation map.  In the equivariant case, the augmentation map is required to be equivariant, which places a severe restriction on the spaces involved.  That is, if $X$ is a $G$--space and $x_0\in X$, then equivariance of the constant map $\epsilon:X\rightarrow\{x_0\}$ requires that $x_0$ be fixed by the $G$--action.  However, if we are willing to restrict our consideration to such spaces, we will still have the equality $\rgenh_\ast^G(X,x_0)=\ker\genh_\ast^G(\epsilon)$, as we now demonstrate in a somewhat more general context.

Suppose that $Y$ is a $H$--space, where $H\leq G$.  Recall (as in \cite{Bredon2}, for instance) that the {\em balanced product} $G\times_HY$ is the quotient of the Cartesian product $G\times Y$ by the $H$--action $h\cdot(g,y)\doteq(gh,h^{-1}y)$; the equivalence class of the point $(g,y)$ is denoted $[g,y]$, and there is a well--defined $G$--action given by $g^\prime\cdot[g,y]=[g^\prime g,y]$.  Moreover, given an equivariant map of $H$--spaces $f:X\rightarrow Y$, we have an induced equivariant map of $G$--spaces $G\times_Hf:G\times_HX\rightarrow G\times_HY$ given by $(G\times_Hf)[g,x]\doteq[g,f(x)]$.

Under suitable hypotheses on $X$ and $G$, it can be show that the existence of an equivariant map $\chi:X\rightarrow Gx_0$ implies that $X$ can be identified with a balanced product $G\times_HY$ for some $H\leq G$ and $H$--space $Y$, and the point $x_0$ corresponds under this identification to a point $[e,y_0]$ with $y_0\in Y$ having trivial $H$--action.  For instance, in the case when $X$ is Hausdorff and $G$ is compact, see propositions 4.1 of chapter I and 3.2 of chapter II in \cite{Bredon2}.

\begin{theorem}
\label{thm:redhombp}
Suppose $H\leq G$, and $(Y,y_0)$ is a pointed $H$--space with $y_0$ having trivial $H$--action.  If $\epsilon_{(Y,y_0)}:Y\rightarrow\{y_0\}$ is the constant map, then
$$\ker\genh_\ast^G(G\times_H\epsilon_{(Y,y_0)})
  \rightarrow
  \rgenh_\ast^G(G\times_HY,[e,y_0])
$$
is an isomorphism.  Moreover, if $f:(Y,y_0)\rightarrow(Z,z_0)$ is a morphism of pointed $H$--spaces, then the following diagram commutes:
$$\begin{CD}
  \ker\genh_\ast^G(G\times_H\epsilon_{(Y,y_0)})
    @>>>\rgenh_\ast^G(G\times_HY,[e,y_0])\\
  @V{\genh_\ast^G(G\times_Hf)}VV
    @VV{\rgenh_\ast^G(G\times_Hf)}V\\
  \ker\genh_\ast^G(G\times_H\epsilon_{(Z,z_0)})
    @>>>
      \rgenh_\ast^G(G\times_HZ,[e,z_0]).
\end{CD}$$
\end{theorem}

\begin{proof}
Set $X\doteq G\times_HY$ and $x_0\doteq[e,y_0]$, and $\pi\doteq G\times_H\epsilon_{(Y,y_0)}$.  The inclusion $i:Gx_0\rightarrow X$ is such that $\pi\circ i={\it id}$.  Therefore, we have an epimorphism of long exact sequences:
$$\begin{CD}
\genh_n^G(Gx_0)
  @>>>
    \genh_n^G(X)
      @>>>
        \genh_n^G(X,Gx_0)
          @>{\D_n}>>
            \genh_{n-1}^G(Gx_0)\\
@V{\genh_n^G(\pi)}VV
  @V{\genh_n^G(\pi)}VV
    @V{\genh_n^G(\pi)}VV
      @V{\genh_{n-1}^G(\pi)}VV\\
\genh_n^G(Gx_0)
  @>>>
    \genh_n^G(Gx_0)
      @>>>
        \genh_n^G(Gx_0,Gx_0)
          @>{\D_n}>>
            \genh_{n-1}^G(Gx_0).
\end{CD}$$
By lemma 8.8 of chapter I in \cite{ES}, the kernel of this sequence is also a long exact sequence.  However, the kernel of the vertical arrow on the each end is trivial, since each map is an isomorphism; the kernel of the vertical arrow second from the left is $\ker\genh_\ast^G(\pi)$, and that of the third arrow from the left is $\genh_\ast^G(X,Gx_0)=\rgenh_\ast^G(X,x_0)$, since $\genh_\ast^G(Gx_0,Gx_0)=0$.

For the second statement of the theorem, we only need to show that $\ker\genh_\ast^G(G\times_Hf)$ maps $\ker\genh_\ast^G(\pi)$ to $\ker\genh_\ast^G(\pi')$, where $\pi'\doteq G\times_H\epsilon_{(Z,z_0)}$.  However, since $f(y_0)=z_0$, we have that $\epsilon_{(Z,z_0)}\circ f=f\circ\epsilon_{(Y,y_0)}$.  Consequently, $\genh_\ast^G(\pi')\circ\genh_\ast^G(G\times_Hf)=\genh_\ast^G(G\times_Hf)\circ\genh_\ast^G(\pi)$.  The theorem follows.
\end{proof}

\section{Associated homology of a filtration}

Let us assume that there exists a {\em $G$--filtration} $\{X_k\}_{k\in\Z}$ of the $G$--pair $(X,A)$; that is, $X=\cup_{k\in\Z}X_k$, where for each $k\in\Z$, $X_k$ is a $G$--subspace of $X$ with $X_k\subseteq X_{k+1}$, and $X_k=A$ if $k<0$.

\begin{lemma}
If $\{X_k\}_{k\in\Z}$ is a $G$--filtration of $(X,A)$, then the following composition is trivial:
$$\genh_n^G(X_p,X_{p-1})
  \xrightarrow{\D_n}
  \genh_{n-1}^G(X_{p-1},X_{p-2})
  \xrightarrow{\D_{n-1}}
  \genh_{n-2}^G(X_{p-2},X_{p-3}).
$$
\end{lemma}

The proof of this lemma is the usual one, see the proof of theorem 6.2 of chapter III of \cite{ES}.  We may thus make the following definition.

\begin{definition}\label{def:asshom}
Let $\{X_k\}_{k\in\Z}$ be a $G$--filtration of the $G$--pair
$(X,A)$.  For each integer $q$, define
\begin{equation}\label{eq:filt-chains}
  \A_qC_p^G(X,A)
  \doteq\genh_{p+q}^G(X_p,X_{p-1}).
\end{equation}
In addition, define
\begin{equation}\label{eq:filt-boundary}
  \A_q\D_p:\A_qC_p^G(X,A)
  \rightarrow\A_qC_{p-1}^G(X,A)
\end{equation}
to be equal to the boundary map of the triple $(X_p,X_{p-1},X_{p-2})$.  The {\bf $q$--th associated equivariant homology of the filtration of $(X,A)$} is then defined to be the homology
\begin{equation}\label{eq:filt-homology}
  \A_q\genh_\ast^G(X,A)
    \doteq H\bigl(\A_qC_\ast^G(X,A),\A_q\D_\ast\bigr).
\end{equation}
 of the chain complex $\bigl(\A_qC_\ast^G(X,A),\A_q\D_\ast\bigr)$.
\end{definition}

Later we will give more explicit formulae for the cells and boundary maps in the case where the filtration is by the skeleta of a simplicial $G$--complex.

\section{Homology spectral sequence}
\label{sec:spectral}

For this section, suppose that we are given a $G$--filtration $\{X_k\}_{k\in\Z}$ of the $G$--pair $(X,A)$.  We define an exact couple $D=D_{\ast,\ast}$ and $E=E_{\ast,\ast}$ of bigraded $R$--modules as follows.  For all $p,q\in\Z$, we set
\begin{equation}\label{eq:genDEdef}
  D_{p,q}\doteq\genh_{p+q}^G(X_p,A)
  \quad\text{and}\quad
  E_{p,q}\doteq\genh_{p+q}^G(X_p,X_{p-1}).
\end{equation}
We also define $R$--module homomorphisms
\begin{equation}\label{eq:genDEmaps}
  i:D_{p,q}\rightarrow D_{p+1,q-1},\quad\quad
  j:D_{p,q}\rightarrow E_{p,q},\quad\quad
  k:E_{p,q}\rightarrow D_{p-1,q},
\end{equation}
where $i$ is induced by $(X_p,A)\rightarrow(X_{p+1},A)$, $j$ by $(X_p,A)\rightarrow(X_p,X_{p-1})$, and $k$ is the boundary map of the triple $(X_p,X_{p-1},A)$.  That $D$ and $E$ form an exact couple is simply a restatement of proposition \ref{prop:triple}.

\begin{lemma}\label{lem:d}
The differential $d\doteq j\circ k$ of the exact couple defined by \eqref{eq:genDEdef} and \eqref{eq:genDEmaps} is equal to the boundary map in the exact sequence of the triple $(X_p,X_{p-1},X_{p-2})$:
$$d:E_{p,q}=\genh_{p+q}^G(X_p,X_{p-1})
      \xrightarrow{\D_{p+q}}
      \genh_{p+q-1}^G(X_{p-1},X_{p-2})=E_{p-1,q}.
$$
\end{lemma}

\begin{proof}
Let $n\doteq p+q$.  From the definition of the boundary map of the triple $(X_p,X_{p-1},A)$, we may decompose $d=j\circ k$ as
$$\genh_n^G(X_p,X_{p-1})
  \xrightarrow{\D_n}
  \genh_{n-1}^G(X_{p-1})
  \rightarrow
  \genh_{n-1}^G(X_{p-1},A)
  \rightarrow
  \genh_{n-1}^G(X_{p-1},X_{p-2}).
$$
By functoriality, the composition of the two arrows on the right is equal to $\genh_{n-1}^G(X_{p-1})\rightarrow\genh_{n-1}^G(X_{p-1},X_{p-2})$; thus the above composition is also equal to the boundary map of the triple $(X_p,X_{p-1},X_{p-2})$.
\end{proof}

Recall (as in \cite{Rotman} chapter 11, for instance) that an exact couple induces a spectral sequence.  Specifically, $D_{p,q}^1\doteq D_{p,q}$ and $E_{p,q}^1\doteq E_{p,q}$, $i^1\doteq i$, $j^1\doteq j$, $k^1\doteq k$, and $d^1\doteq j^1\circ k^1$; and for $r>1$, $D_{p,q}^r\doteq i(D_{p-1,q+1}^{r-1})$,
\begin{equation}\label{eq:genDEr}
  E_{p,q}^r\doteq
  \frac{\ker\{d^{r-1}:E_{p,q}^{r-1}
        \rightarrow E_{p-r+1,q+r-2}^{r-1}\}}
  {\im\{d^{r-1}:E_{p+r-1,q-r+2}^{r-1}
        \rightarrow E_{p,q}^{r-1}\}},
\end{equation}
$i^r$ is induced by $i^{r-1}$, $j^r$ is induced by $j^{r-1}\circ i^{-1}$, $k^r$ is induced by $k^{r-1}$, and $d^r\doteq j^r\circ k^r$.  In particular, $i^r$, $j^r$, $k^r$, and $d^r$ have respective bi--degrees $(1,-1)$, $(1-r,r-1)$, $(-1,0)$, and $(-r,r-1)$.

We will show that the spectral sequence induced by the exact couple in \eqref{eq:genDEdef} and \eqref{eq:genDEmaps} converges to the filtration in homology defined by
\begin{equation}\label{eq:genfiltration}
  \Phi^s\genh_n^G(X,A)
  \doteq
  \im\{\genh_n^G(X_s,A)\rightarrow\genh_n^G(X,A)\}
\end{equation}
for all integers $s$.  However, we will need to make a mild assumption.

\begin{definition}
A $G$--filtration $\{X_p\}_{p\in\Z}$ of the $G$--pair $(X,A)$ is {\bf homologically stable} with respect to the generalized equivariant homology theory $(\genh_\ast^G,\D_\ast)$ if for every $n\in\Z$, there exists $p_n\in\Z$ such that $\genh_n^G(X_p,A)\rightarrow\genh_n^G(X,A)$ is an isomorphism for all integers $p\geq p_n$.
\end{definition}

If the $G$--filtration is stable, that is, if $X_s=X$ for sufficiently large $s$, then it is evidently homologically stable.

\begin{lemma}\label{lem:genfiltration}
If $\{X_p\}_{p\in\Z}$ is homologically stable, then $\{\Phi^s\genh_n^G(X,A)\mid s\in\Z\}$ is a bounded filtration of $\genh_n^G(X,A)$.
\end{lemma}

\begin{proof}
The filtration is bounded below: $\Phi^s\genh_n^G(X,A)=0$ for $s<0$, since $X_s=A$ for such $s$.  It is also bounded above: $\Phi^s\genh_n^G(X,A)=\genh_n^G(X,A)$ for all sufficiently large $s$, by hypothesis.
\end{proof}

\begin{theorem}
\label{thm:genspec}
Suppose that $\{X_p\}_{p\in\Z}$ is a homologically stable $G$--filtration of $X$.  The spectral sequence induced by the exact couple in \eqref{eq:genDEdef} and \eqref{eq:genDEmaps} converges to the filtration \eqref{eq:genfiltration} of $\genh_\ast^G(X,A)$, and we have $E_{p,q}^2=\A_q\genh_p^G(X,A)$.
\end{theorem}

\begin{proof}
The identification of the $E^2$ term and the associated homology of a filtration is implied by lemma \ref{lem:d}.  For convergence, the basic argument is a modification of that of theorem 11.13 in \cite{Rotman}.  Specifically, fix $p,q$ and set $n\doteq p+q$.  Consider the exact sequence (for any $r$):
\begin{equation}\label{eq:gener}
  D_{p+r-2,q-r+2}^r\xrightarrow{i^r}
  D_{p+r-1,q-r+1}^r\xrightarrow{j^r}
  E_{p,q}^r\xrightarrow{k^r}
  D_{p-1,q}^r.
\end{equation}
Since $D_{p+r-1,q-r+1}^r=i^{(r-1)}(D_{p,q})=\im\{\genh_n^G(X_p,A)\rightarrow\genh_n^G(X_{p+r-1},A)\}$, we have that $D_{p+r-1,q-r+1}^r\cong\Phi^p\genh_n^G(X,A)$ for sufficiently large $r$.  The same argument also gives $D_{p+r-2,q-r+2}^r\cong\Phi^{p-1}\genh_n^G(X,A)$ for sufficiently large $r$.  Moreover, since $X_s=A$ for $s<0$, we also have $D_{p-1,q}^r=i^{(r-1)}(D_{p-r,q+r-1})=0$ for sufficiently large $r$.  Consequently, the exact sequence in \eqref{eq:gener} implies that $E_{p,q}^r\cong\Phi^p\genh_n^G(X,A)/\Phi^{p-1}\genh_n^G(X,A)$.

Next, we show that for fixed $n=p+q$ and sufficiently large $r$, $E_{p,q}^r=E_{p,q}^{r+1}$; so that $E_{p,q}^\infty=E_{p,q}^r$.  Indeed for large $r$, $k^r(E_{p,q}^r)=0$, by the above reasoning; whence $d^r(E_{p,q}^r)=0$.  Furthermore, we have the exact sequence
$$D_{p+2r-2,q-2r+3}^r\xrightarrow{i^r}
  D_{p+2r-1,q-2r+2}^r\xrightarrow{j^r}
  E_{p+r,q-r+1}^r.
$$
Again by the above reasoning, we have $D_{p+2r-2,q-2r+3}^r\cong\Phi^{p+r-1}\genh_{n+1}^G(X,A)$, and $D_{p+2r-1,q-2r+2}^r\cong\Phi^{p+r}\genh_{n+1}^G(X,A)$; thus boundedness of the filtration implies that $i^r$ is an isomorphism for large $r$.  Therefore, $j^r(D_{p+2r-1,q-2r+2}^r)=0$; and thus $d^r(E_{p+r,q-r+1}^r)=0$.  Consequently, $E_{p,q}^{r+1}=E_{p,q}^r$.
\end{proof}

\section{Equivariant singular homology}

The equivariant generalization of singular homology is defined in \cite{Brocker} and in \cite{Illman}.  We give a brief account of its construction, albeit with modified notation.

\subsection{Coefficient systems}

As in \cite{Bredon1}, one defines the category of {\em canonical $G$--orbits} to be the category ${\cal O}(G)$ whose objects are the left cosets of $G$, and whose morphisms are equivariant maps between these cosets.  From \cite{me}, we have the following.

\begin{proposition}\label{prop:cont}
Every morphism of ${\mathcal O}(G)$ is continuous.
\end{proposition}

In the following, we will make use of the following morphisms in ${\mathcal O}(G)$:  if $H\leq G$ and $g\in G$, we define
$$\mu(g,H):G/H^g\rightarrow G/H,\quad
  \mu(g,H)(eH^g)\doteq gH
$$
(extended equivariantly).  Here $H^g$ denotes the conjugate of $H$ by $g$: $H^g\doteq\{ghg^{-1}\mid h\in H\}$.  When the subgroup $H$ can be discerned from context, we will write $\mu_g$ instead of $\mu(g,H)$.  Also, if $H\leq K\leq G$, then we define
$$\kappa(K,H):G/H\rightarrow G/K,\quad
  \kappa(K,H)(eH)\doteq eK
$$
(extended equivariantly).  When the subgroups $H,K$ can be deduced from context, we simply write $\kappa$ instead of $\kappa(K,H)$.

\begin{lemma}\label{lem:mu-kappa}
The maps $\mu(g,H)$ and $\kappa(K,H)$ are well--defined, equivariant, continuous, and satisfy (i) if $gH=g'H$, then $\mu(g,H)=\mu(g',H)$, (ii) $\mu(g,H)$ is a homeomorphism with $\mu(g,H)^{-1}=\mu(g^{-1},H^g)$, (iii) if $H\leq K\leq L\leq G$, then $\kappa(L,K)\circ\kappa(K,H)=\kappa(H,L)$, (iv) $\mu(g,H)\circ\mu(g',H^{g'})=\mu(g'g,H)$, and (v) if $H\leq K$, then $\mu(g,K)\circ\kappa(K^g,H^g)=\kappa(K,H)\circ\mu(g,H)$.
\end{lemma}

\begin{proof}
Since $(ghg^{-1})gH=gH$, $\mu(g,H)$ is well--defined; and since $H\leq K$, $\kappa(K,H)$ is also well--defined.  Both maps are a priori equivariant, and by proposition \ref{prop:cont}, continuous.  Properties (i) and (iv) follow from the fact $H^{g_1g_2}=(H^{g_2})^{g_1}$; the other properties follow by computation.
\end{proof}

A {\em covariant $R$--coefficient system} for $G$ is a covariant functor $\M$ from ${\cal O}(G)$ to the category of left $R$--modules and linear maps.  Note that if $H\leq G$, then $\M(\mu_g):\M(G/H^g)\rightarrow\M(G/H)$ is an isomorphism for any $g\in G$.  Now given a $G$--set $S$, the orbit $[s]\doteq Gs$ of $s\in S$ can be identified setwise with the the coset $G/G_s$, where $G_s$ is the stabilizer subgroup of $s$.  However, this identification depends on the choice of orbit representative $s$; indeed, if $s'=gs$, then $G_{s'}=G_s^g$.  To avoid this dependence, we make the following definition.

\begin{definition}
\label{def:orbitmodule}
Let $S$ be a $G$--set and $\M$ a covariant $R$--coefficient system for $G$.  For $s\in S$, we define the {\bf orbit module} to be
$$\M[s]
    \doteq
  \left(\bigoplus_{gG_s\in G/G_s}\M(G/G_s^g)
        \otimes_RR\{gs\}\right)
    /J_\M[s],$$
where $J_\M[s]$ is defined to be the submodule generated by all elements of the form $x\otimes gs-\M\bigl(\mu(g,G_s)\bigr)(x)\otimes s$, where $g\in G$ and $x\in\M(G/G_s^g)$.
\end{definition}

\noindent
Here $R\{gs\}$ denotes the free $R$--module with generator $gs\in S$; its usage is merely to provide a convenient indexing scheme.  We use square brackets to denote elements of $\M[s]$; i.e., we write $[x\otimes gs]$ instead of $x\otimes gs+J_\M[s]$.

\begin{lemma}
\label{lem:genindep}
If $[s']=[s]$, then $\M[s']=\M[s]$.
\end{lemma}

\begin{proof}
Suppose $s'=gs$.  It suffices to show $J_\M[s']=J_\M[s]$, for which in turn, it suffices to show $J_\M[s']\subseteq J_\M[s]$.  Define $\omega(x_1,g_1)\doteq x_1\otimes g_1s-\M(\mu_{g_1})(x_1)\otimes s\in J_\M[s]$.  One then computes that $x'\otimes g's'-\M(\mu_{g'})(x')\otimes s'=\omega(x',g'g)-\omega(\M(\mu_{g'})(x'),g)$.
\end{proof}

\begin{lemma}
\label{lem:orbitmodule}
The map $\M(G/G_s^g)\rightarrow\M[s]$ given by $x\mapsto[x\otimes gs]$ is an isomorphism for any $g\in G$ and $s\in S$.
\end{lemma}

\begin{proof}
By lemma \ref{lem:genindep}, we may assume $g=e$.  Injectivity follows from the fact that $x\otimes s\not\in J_\M[s]$ unless $x=0$.  Surjectivity follows from the identity $[x'\otimes g's]=[\M(\mu_{g'H})(x')\otimes s]$.
\end{proof}

The orbit module allows us to extend a given covariant coefficient system $\M$ for $G$ to a covariant functor from the category of $G$--sets and equivariant maps to the category of left $R$--modules and linear maps.

\begin{definition}
\label{def:coeffext}
Suppose $\M$ is a covariant $R$--coefficient system for $G$.  Define $\M(\emptyset)=0$.  For a $G$--set $S\neq\emptyset$, define
$$\M(S)\doteq\bigoplus_{[s]\subseteq S}\M[s]$$
(the sum is over orbits $[s]$ of $S$).  Moreover, if $S,S^\prime$ are $G$--sets and $f:S\rightarrow S^\prime$ is equivariant, then define
$$\M(f):\M(S)\rightarrow\M(S^\prime)$$
by extending $\M(f)([x\otimes s])\doteq[\M\bigl(\kappa(G_{f(s)},G_s)\bigr)(x)\otimes f(s)]$ linearly, where $s\in S$ and $x\in\M(G/G_s)$.
\end{definition}

That $\M(f)$ is independent of the choice of orbit representative $s$ follows from lemma \ref{lem:mu-kappa} and functoriality of $\M$, as does the following.

\begin{theorem}
Suppose $\M$ is a covariant coefficient system.  If $S$ is a $G$--set, then $\M({\it id}_S)$ is the identity.  Moreover, if $f:S\rightarrow S'$ and $f^\prime:S'\rightarrow S''$ are equivariant maps of $G$--sets, then $\M(f'\circ f)=\M(f')\circ\M(f)$.\hfill\qed
\end{theorem}

A generalized equivariant homology theory $(\genh_\ast^G,\D_\ast)$ induces a covariant coefficient system by restricting $\genh_\ast^G$ to ${\cal O}(G)$.  However if $X$ is a $G$--space and $x\in X$, we potentially need to distinguish between the homology module $\genh_\ast^G([x])$ and the orbit module $\genh_\ast^G[x]$.  These two modules are isomorphic in the case when $G$ is a compact Lie group and $X$ is Hausdorff, for then the quotient space $G/G_x$ and the subspace $[x]$ are equivariantly homeomorphic (see proposition 4.1 of chapter 1 in \cite{Bredon2}), and lemma \ref{lem:orbitmodule} applies.  On the other hand, if $G$ is not compact Lie, then the two modules need not coincide.

\subsection{Equivariant singular homology}

Given a space $X$, we let $S_n(X)$ denote the set of all singular $n$--simplices: the set of all (continuous) maps $\sigma:\Delta_n\rightarrow X$, where $\Delta_n$ denotes the standard $n$--simplex.  In particular when $n<0$, we have $S_n(X)=\emptyset$.  If $X$ is a $G$--space, then $S_n(X)$ is a $G$--set, where $g\sigma:\Delta_n\rightarrow X$ is the map $(g\sigma)(x)\doteq g(\sigma(x))$.

Let $\Delta_n^{(k)}$ denote the $k$--th face of $\Delta_n$; and if $\sigma$ is a singular $n$--simplex, we let $\sigma^{(k)}:\Delta_{n-1}\rightarrow X$ denote the composition of the identification map $\Delta_{n-1}\rightarrow\Delta_n^{(k)}$ followed by $\sigma$.  Each map $\D_n^{(k)}:S_n(X)\rightarrow S_{n-1}(X)$, with $\D_n^{(k)}(\sigma)\doteq\sigma^{(k)}$, is equivariant; so we may make the following definition.

\begin{definition}\label{def:sing-chains}
If $X$ is a $G$--space and $\M$ is a covariant coefficient system, we define for each integer $n$
$$S_n^G(X;\M)\doteq\M(S_n(X)),$$
and the graded $R$--module $S_\ast^G(X;\M)\doteq\oplus_{n\in\Z}S_n(X;\M)$.  In addition, we define the map
$$\D_n:S_n^G(X;\M)\rightarrow S_{n-1}^G(X;\M)$$
by $\D_n\doteq\sum_{k=0}^n(-1)^k\M(\D_n^{(k)})$.
\end{definition}

As in the nonequivariant case, one verifies that $\D_{n-1}\circ\D_n=0$.  Moreover, if $(X,A)$ is a $G$--pair, then $S_n^G(A;\M)$ is necessarily a summand of $S_n^G(X;\M)$, and $\D_n(S_n^G(A;\M))\subseteq S_{n-1}^G(A;\M)$.

\begin{definition}\label{def:sing-hom}
If $(X,A)$ is a $G$--pair, and $\M$ is a covariant coefficient system, we define $S_\ast^G(X,A;\M)\doteq S_\ast^G(X;\M)/S_\ast^G(A;\M)$, and we define $\D_\ast:S_\ast^G(X,A;\M)\rightarrow S_{\ast-1}^G(X,A;\M)$ to be the map induced by $\D_\ast:S_\ast^G(X;\M)\rightarrow S_{\ast-1}^G(X;\M)$ under this quotient.  We then define the {\bf equivariant singular homology} of $(X,A)$ to be the $R$--module
$$H_\ast^G(X,A;\M)\doteq H(S_\ast^G(X,A;\M),\D_\ast);$$
i.e., the homology of the chain complex $(S_\ast^G(X,A;\M),\D_\ast)$.
\end{definition}

In particular, $H_n^G(X,A;\M)$ is trivial for $n<0$.  From \cite{Brocker} (in the case $G$ is discrete) and \cite{Illman} (in the general case), we have the following.

\begin{proposition}
\label{prop:singhom}
Equivariant singular homology $(H_\ast^G,\D_\ast)$ with coefficients in a covariant coefficient system $\M$ is a generalized equivariant homology theory.  Moreover when restricted to ${\cal O}(G)$, $H_n^G$ is trivial if $n\neq 0$, and can be naturally identified with $\M$ otherwise.
\end{proposition}

\noindent
That is, there are isomorphisms $i_H:H_0^G(G/H;\M)\rightarrow\M(G/H)$ for each $H\leq G$ such that if $\eta:G/H\rightarrow G/K$ is equivariant, then $\M(\eta)\circ i_H=i_K\circ H_0^G(\eta)$.

\begin{theorem}
Equivariant singular homology satisfies the arbitrary sum property.
\end{theorem}

\begin{proof}
Let $\{(X_\alpha,A_\alpha)\}_{\alpha\in\A}$ be a collection of $G$--pairs, and set $X\doteq\sum_{\alpha\in\A}X_\alpha$ and $A\doteq\sum_{\alpha\in\A}A_\alpha$.  From definitions \ref{def:coeffext}, \ref{def:sing-chains}, and \ref{def:sing-hom}, we may make the identification $S_\ast^G(X,A;\M)=\oplus_{\alpha\in\A}S_\ast^G(X_\alpha,A_\alpha;\M)$; and under this identification, $\D_\ast=\oplus_{\alpha\in\A}\D_\ast^\alpha$, where $\D_\ast^\alpha\doteq\D_\ast:S_\ast^G(X_\alpha,A_\alpha;\M)\rightarrow S_{\ast-1}^G(X_\alpha,A_\alpha;\M)$.
\end{proof}

\begin{theorem}
If $G$ is compact, then equivariant singular homology has compact supports.
\end{theorem}

\begin{proof}
Let $(X,A)$ be a $G$--pair.  Note that $G/H$ is compact for any $H\leq G$, and the orbit under $G$ of the image of any singular simplex of $X$ is also compact.  For $y\in S_\ast^G(X;\M)$, let $X(y)$ denote the union of all orbits $G\cdot{\rm im}(\sigma)$, where $[x_\sigma\otimes\sigma]$ is a nonzero summand of $y$ for some $\sigma\in S_\ast(X)$ and $x_\sigma\in\M(G/G_\sigma)$; observe that $X(y)$ is compact.  Given $x\in H_n^G(X,A;\M)$, represent $x$ by some $y\in S_n^G(X;\M)$ and take $X_c\doteq X(y)$, $A_c\doteq X(\D{y})$.
\end{proof}

\section{Homology of an equivariant simplex}

Most of the ideas and proofs in this section are equivariant generalizations of those given in \cite{ES}, chapter III, sections 3 and 4.

Let $s$ denote an ordered $p$--simplex with vertices $v_0<v_1<\cdots<v_p$.  We let $s_k$ denote the $k$--th face of $s$: the ordered $(p-1)$--simplex with the vertex $v_k$ omitted.  Also, we let $\hat{s}_k$ denote the simplicial subcomplex of $s$ obtained by removing the interior of $s_k$ from $\D s$; that is, $\hat{s}_k$ is the union of all $l$--faces of $s$, with $l\neq k$.  The $G$--action on the {\em equivariant} simplex $G/H\times s$ (with $H\leq G$) will always be on the first factor only: $g\cdot(g'H,x)=(gg'H,x)$.

\begin{lemma}
\label{lem:simpincidence}
If $H\leq G$, $p\geq 1$, and $0\leq k\leq p$, then the map
$$[s:s_k]_\ast^H:
    \genh_\ast^G(G/H\times s,G/H\times\D s)
      \rightarrow
    \genh_{\ast-1}^G(G/H\times s_k,G/H\times\D s_k),
$$
defined as the composition
\begin{multline*}
  \genh_\ast^G(G/H\times s,G/H\times\D s)
    \xrightarrow{\D_\ast}
  \genh_{\ast-1}^G(G/H\times\D s,G/H\times\hat{s}_k)\\
    \leftarrow
  \genh_{\ast-1}^G(G/H\times s_k,G/H\times\D s_k),
\end{multline*}
is an isomorphism.
\end{lemma}

\begin{proof}
$(G/H\times s_k,G/H\times\D s_k)\rightarrow(G/H\times\D s,G/H\times\hat{s}_k)$ induces an isomorphism in homology by strong excision.  Moreover, since $G/H\times\hat{s}_k$ is an equivariant strong deformation retract of $G/H\times s$, $\D_\ast$ is an isomorphism by the long exact sequence of the triple $(G/H\times s,G/H\times\D s,G/H\times\hat{s}_k)$.
\end{proof}

\begin{lemma}
\label{lem:simpincidence-nat}
Suppose $f:s\rightarrow s'$ is a simplicial map of $p$--simplices and $\eta:G/H\rightarrow G/K$ is equivariant.  If there exists $k,l$ with $f(\hat{s}_k)=\hat{s}_l^\prime$ and $f(s_k)=s_l^\prime$, then $\genh_{\ast-1}^G(\eta\times f)\circ[s:s_k]_\ast^H=[s':s_l^\prime]_\ast^K\circ\genh_\ast^G(\eta\times f)$.
\end{lemma}

\begin{proof}
This follows from lemma \ref{lem:simpincidence} and the following commutative diagram.
$$\begin{CD}
  \genh_\ast^G(G/H\times s,G/H\times\D s)
    @>{\genh_\ast^G(\eta\times f)}>>
      \genh_\ast^G(G/K\times s',G/K\times\D s')\\
  @V{\D_\ast}VV
    @V{\D_\ast}VV\\
  \genh_{\ast-1}^G(G/H\times\D s,G/H\times\hat{s}_k)
    @>{\genh_{\ast-1}^G(\eta\times f)}>>
            \genh_{\ast-1}^G(G/K\times\D s',
                             G/K\times\hat{s}_l^\prime)\\
  @AAA
    @AAA\\
  \genh_{\ast-1}^G(G/H\times s_k,G/H\times\D s_k)
    @>{\genh_{\ast-1}^G(\eta\times f)}>>
      \genh_{\ast-1}^G(G/K\times s_l^\prime,
                       G/K\times\D s_l^\prime).
\end{CD}$$
Indeed, the square on the top commutes by naturality of the long exact sequence of a triple, and the square on the bottom by functoriality.
\end{proof}

\begin{theorem}
\label{thm:simp-id}
Suppose $H\leq G$, $s$ is a $p$--simplex, and $\pi:G/H\times s\rightarrow G/H$ is the projection map.  The map
$$\zeta(H,s)_\ast:
  \genh_\ast^G(G/H\times s,G/H\times\D s)
    \rightarrow
  \genh_{\ast-p}^G(G/H)
$$
defined inductively by
\begin{equation}\label{eq:simp-iso}
  \zeta(H,s)_\ast
  \doteq
  \begin{cases}
    \genh_\ast^G(\pi)
      &\text{if $p=0$}\\
    (-1)^p\zeta(H,s_p)_{\ast-1}\circ[s:s_p]_\ast^H
      &\text{if $p>0$}
  \end{cases}
\end{equation}
is an isomorphism.
\end{theorem}

\begin{proof}
When $p=0$, $\pi$ is an equivariant homeomorphism; the result then follows by induction and lemma \ref{lem:simpincidence}.
\end{proof}

\begin{lemma}
\label{lem:id-nat}
If $f:s\rightarrow s'$ is an order--preserving simplicial map of $p$--simplices and $\eta:G/H\rightarrow G/K$ is equivariant, then $\zeta(K,s')_\ast\circ\genh_\ast^G(\eta\times f)=\genh_{\ast-p}^G(\eta)\circ\zeta(H,s)_\ast$.
\end{lemma}

\begin{proof}
For $p=0$, the statement of the theorem follows by functoriality and the identity $\pi\circ(\eta\times f)=\eta\circ\pi$.  For $p>0$, $f(s_p)=s_p^\prime$ and $f(\hat{s}_p)=\hat{s}_p^\prime$, as $f$ is order--preserving; induction in conjunction with lemma \ref{lem:simpincidence-nat} and equation \eqref{eq:simp-iso} thus yield the theorem.
\end{proof}

\begin{lemma}\label{lem:refl}
If $\rho:s\rightarrow s$ is a simplicial map that exchanges exactly two vertices, then $\genh_\ast^G({\it id}\times\rho):\genh_\ast^G(G/H\times s,G/H\times\D s)\rightarrow\genh_\ast^G(G/H\times s,G/H\times\D s)$ is multiplication by $-1$.
\end{lemma}

\begin{proof}
Consider the case $p=1$; that is, $\D s=\{v_0,v_1\}$.  For $k=0,1$, let $i_k:G/H\rightarrow G/H\times\D s$ be the map $i_k(gH)\doteq(gH,v_k)$, and $\pi_k:G/H\times\D s\rightarrow G/H\times\{v_k\}$ the map $\pi_k(gH,v)\doteq(gH,v_k)$.  Now by proposition \ref{prop:findisj}, we have an isomorphism $S_\ast:\genh_\ast^G(G/H)\oplus\genh_\ast^G(G/H)\rightarrow\genh_\ast^G(G/H\times\D s)$ given by $S_\ast(x,y)=\genh_\ast^G(i_0)(x)+\genh_\ast^G(i_1)(y)$.  Since $\pi_k\circ i_{k'}=i_k$, one computes that $\genh_\ast^G(\pi_k)\circ S_\ast(x,y)=\genh_\ast^G(i_k)(x+y)$.  Thus $\ker\genh_\ast^G(\pi_k)$ is the image of $\{(x,-x)\mid x\in\genh_\ast^G(G/H)\}$ under $S_\ast$.  Further, since $({\it id}\times\rho)\circ i_k=i_{k'}$, where $k'=1,0$ when $k=0,1$, one computes $\genh_\ast^G({\it id}\times\rho)\circ S_\ast(x,-x)=-S_\ast(x,-x)$.  That is, the vertical arrow on the left in the diagram
$$\begin{CD}
  \ker\genh_{\ast-1}^G(\pi_0)
    @>{(\cong)}>>
      \rgenh_{\ast-1}^G(G/H\times\D s,v_0)
        @<{\rD_\ast}<<
          \genh_\ast^G(G/H\times s,G/H\times\D s)\\
  @VV{\genh_{\ast-1}^G({\it id}\times\rho)}V
    @VV{\genh_{\ast-1}^G({\it id}\times\rho)}V
      @VV{\genh_{\ast-1}^G({\it id}\times\rho)}V\\
  \ker\genh_{\ast-1}^G(\pi_1)
    @>{(\cong)}>>
      \rgenh_{\ast-1}^G(G/H\times\D s,v_1)
        @<{\rD_\ast}<<
          \genh_\ast^G(G/H\times s,G/H\times\D s)
\end{CD}$$
is multiplication by $-1$.  The above diagram commutes: the square on the left by theorem \ref{thm:redhombp} (since $s$ has trivial $H$--action, $G\times_Hs=G/H\times s$), the square on the right by theorem \ref{thm:redhom-nat}.  However, since $G/H\times s$ is contractible, $\rD_\ast$ is an isomorphism; and the lemma follows when $p=1$.

In the case $p>1$, there is a vertex $v_k$ such that $\rho(v_k)=v_k$.  Therefore, $\rho(\hat{s}_k)=\hat{s}_k$ and $\rho(s_k)=s_k$.  Thus by lemma \ref{lem:simpincidence-nat}, $\genh_{\ast-1}^G(G/H\times\rho)\circ[s:s_k]_\ast^H=[s:s_k]_\ast^H\circ\genh_\ast^G(G/H\times\rho)$, and $-[s:s_k]_\ast^H=[s:s_k]_\ast^H\circ\genh_\ast^G(G/H\times\rho)$ follows inductively.  Since $[s:s_k]_\ast^H$ is an isomorphism, we are done.
\end{proof}

\begin{theorem}
\label{thm:simp-id-perm}
If $s$ is a $p$--simplex, $\phi:s\rightarrow s$ is a simplicial map that permutes the vertices, and $\eta:G/H\rightarrow G/K$ is $G$--equivariant, then $\zeta(K,s)_\ast\circ\genh_\ast^G(\eta\times\phi)=(\sign\phi)\genh_{\ast-p}^G(\eta)\circ\zeta(H,s)_\ast$.
\end{theorem}

\begin{proof}
Let $i:s\rightarrow s$ and $j:G/H\rightarrow G/H$ denote the identity maps.  Since $\eta\times\phi=(\eta\times i)\circ(j\times\phi)$, by lemma \ref{lem:id-nat}, $\zeta(K,s)_\ast\circ\genh_\ast^G(\eta\times\phi)=\zeta(K,s)_\ast\circ\genh_\ast^G(\eta\times i)\circ\genh_\ast^G(j\times\phi)=\genh_{\ast-p}^G(\eta)\circ\zeta(H,s)_\ast\circ\genh_\ast^G(j\times\phi)$.  However, $\phi$ is the composition of simplicial reflections, say $\phi=\rho_1\circ\cdots\circ\rho_n$; and so $j\times\phi=(j\times\rho_1)\circ\cdots\circ(j\times\rho_n)$.  Hence by lemma \ref{lem:refl} and functoriality, $\genh_\ast^G(j\times\phi)=(-1)^n=\sign\phi$.  The theorem now follows.
\end{proof}

\begin{theorem}
\label{thm:id-nat-perm}
Suppose $s,s'$ are ordered $p$--simplices with respective vertices $v_1<\cdots<v_p$ and $v_1'<\cdots<v_p'$, and that $f:s\rightarrow s'$ is a simplicial map such that $f(v_k)=v_{\phi(k)}'$, $0\leq k\leq p$, for some permutation $\phi$.  Then $\zeta(K,s')_\ast\circ\genh_\ast^G(\eta\times f)=(\sign\phi)\,\genh_{\ast-p}^G(\eta)\circ\zeta(H,s)_\ast$ for any equivariant map $\eta:G/H\rightarrow G/K$.
\end{theorem}

\begin{proof}
Define $\chi:s\rightarrow s$ and $\iota:s\rightarrow s'$ to be the simplicial maps such that $\chi(v_k)=v_{\phi(k)}$ and $\iota(v_k)=v_k'$, for $0\leq k\leq p$.  Then $\eta\times f=(i\times\iota)\circ(\eta\times\chi)$, where $i:G/K\rightarrow G/K$ is the identity, $\chi$ permutates the vertices of $s$ with $\sign\chi=\sign\phi$, and $\iota$ is order preserving.  The statement of the theorem follows now from functoriality, lemma \ref{lem:id-nat}, and theorem \ref{thm:simp-id-perm}.
\end{proof}

\begin{theorem}
\label{thm:simp-perm-id}
If $H\leq G$ and $s$ is a $p$--simplex, then $\zeta(H,s_k)_{\ast-1}\circ[s:s_k]_\ast^H=(-1)^k\zeta(H,s)_\ast$.
\end{theorem}

\begin{proof}
If $k=p$, then this is just equation \eqref{eq:simp-iso}, so we can assume that $k<p$.  Let $\rho_j$ denote the reflection of $s$ that exchanges vertices $v_j$ and $v_{j+1}$, but leaves all other vertices fixed ($0\leq j<p$); and set $\phi\doteq\rho_k\circ\rho_{k+1}\circ\cdots\circ\rho_{p-1}$, so that $\phi(v_i)=v_i$ if $0\leq i<k$, $\phi(v_k)=v_p$, and $\phi(v_i)=v_{i-1}$ if $k<i\leq p$.  Note that $\sign\phi=(-1)^{p-k}$, $\phi(s_k)=s_p$, and $\phi(\hat{s}_k)=\hat{s}_p$.  The diagram
$$\begin{CD}
  \genh_\ast^G(G/H\times s,G/H\times\D s)
    @>{[s:s_k]_\ast^H}>>
      \genh_{\ast-1}^G(G/H\times s_k,G/H\times\D s_k)\\
  @V{\genh_\ast^G({\it id}\times\phi)}VV
    @V{\genh_{\ast-1}^G({\it id}\times\phi)}VV\\
  \genh_\ast^G(G/H\times s,G/H\times\D s)
    @>{[s:s_p]_\ast^H}>>
      \genh_{\ast-1}^G(G/H\times s_p,G/H\times\D s_p)\\  
  @V{\zeta(H,s)_\ast}VV
    @V{\zeta(H,s_p)_{\ast-1}}VV\\
  \genh_{\ast-p}^G(G/H)
    @>{(-1)^p}>>
      \genh_{\ast-p}^G(G/H)
\end{CD}$$
commutes: the top square by lemma \ref{lem:simpincidence-nat}, and the bottom square by equation \eqref{eq:simp-iso}.  Now by theorem \ref{thm:simp-id-perm}, the two vertical arrows on the left can be identified with $(-1)^{p-k}\zeta(H,s)_\ast$.  Moreover, the simplicial map $\phi:s_k\rightarrow s_p$ is order--preserving, so by lemma \ref{lem:id-nat}, the vertical arrows on the right can be identified with $\zeta(H,s_k)_{\ast-1}$.  Consequently, $\zeta(H,s_k)_{\ast-1}\circ[s:s_k]_\ast^H=(-1)^p(-1)^{p-k}\zeta(H,s)_\ast=(-1)^k\zeta(H,s)_\ast$.
\end{proof}

\section{Homology of a simplicial complex}

For the remainder, we assume that $G$ is a locally compact Hausdorff topological group.

\begin{definition}
Let $V$ be a $G$--set.  An {\bf abstract simplicial $G$--complex} is a collection of finite nonempty subsets ${\mathfrak S}$ of $V$ satisfying (i) if $v\in V$, then $\{v\}\in{\mathfrak S}$, (ii) if $S\in{\mathfrak S}$ and $R\subseteq S$, then $R\in{\mathfrak S}$, (iii) if $S\in{\mathfrak S}$, then $gS\in{\mathfrak S}$ for all $g\in G$, and (iv) if $S\in{\mathfrak S}$ and $g\in G_S$, then $gx=x$ for all $x\in S$.
\end{definition}

\noindent
Note that an abstract simplicial $G$--complex is a (nonequivariant) abstract simplicial complex via conditions (i) and (ii).

We recall the standard construction for the realization of an abstract simplicial complex in the nonequivariant case.  Let $\R{V}$ denote the real vector space with basis $V$, and we give it the co--limit topology induced from the metric topology on finite subspaces.  For each abstract simplex $S\in{\mathfrak S}$, we form the convex hull $|S|\subseteq\R{V}$ of $S$.  The realization of ${\mathfrak S}$ is then the union $|{\mathfrak S}|\doteq\cup_{S\in{\mathfrak S}}|S|$, but with the coherent topology with respect to the constituent simplices.  Observe that the $G$--action on $V$ extends linearly to a $G$--action on $|{\mathfrak S}|$.  However, this action will not be continuous unless $G$ is discrete.  On the other hand, using this construction, we may now identify an abstract simplex $S\in{\mathfrak S}$ with its realization $|S|\subseteq\R{V}$.  In the following, we will use lower case Roman letters to emphasize that we have made this identification; i.e., we write $s\in{\mathfrak S}$ to mean $s=|S|$ for some $S\in{\mathfrak S}$.

In \cite{me}, a construction is given for the realization of an abstract simplicial $G$--complex, where $G$ is any locally compact Hausdorff topological group; in the case when $G$ is finite, the resulting space is equivariantly homeomorphic to the above realization.  We give a brief outline of the construction.  As a set, the realization ${\it top}_G({\mathfrak S})$ is the quotient of the disjoint union $\sum_{s\in{\mathfrak S}}G/G_s\times s$ under the following equivalence relation.  For $s,t\in{\mathfrak S}$, $(g_sG_s,x_s)\in G/G_s\times s$ and $(g_tG_t,x_t)\in G/G_t\times t$ are identified if there exists $g\in G$ and $r\in{\mathfrak S}$ such that $x_s\in r\subseteq s$, $x_t\in gr\subseteq t$, and $g_s^{-1}g_tg\in G_r$.  Let $\chi_s:G/G_s\times s\rightarrow{\it top}_G({\mathfrak S})$ be the map induced by inclusion; the image of $\chi_s$ is given the quotient topology, and ${\it top}_G({\mathfrak S})$ is given the coherent topology with respect to all such images.

In this section, we will require the services of the following properties of ${\it top_G}(\mathfrak{S})$.
\begin{enumerate}
  \item With the exception of the Hausdorff axiom\footnote{If $G_s$ is not a closed subset of $G$, which it is not assumed to be, $G/G_s$ is not Hausdorff.}, ${\it top}_G({\mathfrak S})$ is an equivariant cw--complex with respect to the maps $\{\chi_s\}_{s\in\mathcal{R}}$, where $\mathcal{R}$ is any collection of orbit representatives for $\mathfrak{S}$.
  \item $\chi_{gs}=\chi_s\circ\bigl(\mu(g,G_s)\times g^{-1}\bigr)$, where $g^{-1}:gs\rightarrow s$ is left multiplication by $g^{-1}$.  Consequently, $\chi_{gs}$ and $\chi_s$ have the same images in ${\it top}_G(\mathfrak{S})$.
  \item $\chi_{s_k}\circ\bigl(\kappa(G_{s_k},G_s)\times{\it id}\bigr)=\chi_s|G/G_s\times s_k$ for each face $s_k$ of $s$.
\end{enumerate}
In addition, if $\mathfrak{R}$ is an abstract simplicial $G$--complex with $\mathfrak{R}\subseteq\mathfrak{S}$, then ${\it top}_G(\mathfrak{R})$ is a $G$--subspace of ${\it top}_G(\mathfrak{S})$.  Details are provided in \cite{me}.

We say that $X$ is a {\bf simplicial $G$--complex} if it is the realization of an abstract simplicial $G$--complex, and we let $C_\ast(X)$ denote the underlying abstract simplicial $G$--complex; i.e., $X={\it top}_G(C_\ast(X))$.  More generally, we say that $(X,A)$ is a {\bf pair} of simplicial $G$--complexes if $X$ and $A$ are simplicial $G$--complexes with $C_\ast(A)\subseteq C_\ast(X)$.  We let $C_p(X)$ denote the collection of all simplices of $X$ of dimension $p$, so that $C_\ast(X)=\cup_{p\in\Z}C_p(X)$.  Note that $C_p(X)$ is a $G$--set, as are $C_p(A)\subseteq C_p(X)$ and $C_p(X,A)\doteq C_p(X)\setminus C_p(A)$.

If $(X,A)$ is a pair of simplicial $G$--complexes, we may filter by equivariant cell dimension.  That is, since it is an equivariant cw--complex, $X$ is partitioned by open equivariant cells $\chi_s\bigl(G/G_s\times{\rm int}(s)\bigr)$ of dimension $\dim(s)$, where $s$ is in any collection of orbit representatives in $C_\ast(X)$.  We let $X^p$ denote the $p$--skeleton of $X$: the collection of all open equivariant cells of dimension $q$ with $q\leq p$.   The desired filtration is then by sets $X_p\doteq X^p\cup A$.

In the following discussion, it will be assumed that for each (nonequivariant) simplex $s$, we have chosen a preferred ordering of its vertices; however, the vertices of the complex to which $s$ belongs is may or not be (globally) ordered.  Also, it is not assumed that the group action preserves the preferred vertex ordering; i.e., the vertex ordering of $gs$, for $g\in G$, induced by that of $s$ is not necessarily the preferred one.  We will let $\sign(g,s)$ denote the sign of any vertex permutation that puts the vertices of $gs$ into preferred order.

\subsection{Alternate description of chains and boundaries}

Throughout this section, we assume that $(\genh_\ast,\D_\ast)$ is a generalized equivariant homology theory {\em that satisfies the arbitrary sum property.}  In the absence of this, all of the results will hold for simplicial complexes that have a finite number of equivariant cells in each dimension.  Throughout, we let $(X,A)$ be a pair of simplicial $G$--complexes.

\begin{lemma}\label{lem:simp-chains}
If ${\cal R}$ is a collection of orbit representatives for $C_p(X,A)$, then
$$\Theta_\ast:
  \bigoplus_{s\in{\cal R}}
    \genh_\ast^G(G/G_s\times s,G/G_s\times\D{s})
      \rightarrow
  \genh_\ast^G(X_p,X_{p-1}),$$
given by $\Theta_\ast(x)=\genh_\ast^G(\chi_s)(x)$ for $x\in\genh_\ast^G(G/G_s\times s,G/G_s\times\D{s})$, is an isomorphism.
\end{lemma}

\begin{proof}
The map $\theta:\sum_{s\in{\cal R}}(G/G_s\times s,G/G_s\times\D{s})\rightarrow(X_p,X_{p-1})$, defined by $\theta|(G/G_s\times s,G/G_s\times\D{s})=\chi_s$, induces an isomorphism in homology.  Indeed, $G/G_s\times\D{s}$ is a equivariant retract of $G/G_s\times(s\setminus\{s_b\})$, where $s_b$ is the barycenter of $s$.  So by excising $G/G_s\times\D{s}$ from $\bigl(G/G_s\times s,G/G_s\times(s\setminus\{s_b\})\bigr)$, $\theta$ is the same in homology as the restriction
\begin{multline*}
  \sum_{s\in\mathcal{R}}\bigl(G/G_s\times\Int(s),
         G/G_s\times(\Int(s)\setminus\{s_b\})\bigr)\\
  \quad\xrightarrow{\theta}
    \bigcup_{s\in\mathcal{R}}\bigl(\chi_s(G/G_s\times\Int(s)),
          \chi_s(G/G_s\times(\Int(s)\setminus\{s_b\}))\bigr)
\end{multline*}
(for more details, see the proof of lemma 39.2 in \cite{Munkres}).  As this map is an equivariant homeomorphism, the lemma now follows from theorem \ref{thm:arbdisj}.
\end{proof}

\begin{theorem}
\label{thm:simplex-chains}
If $(X,A)$ is a pair of simplicial $G$--complexes, then the map
$$\alpha_{p,q}:
  \genh_q^G\bigl(C_p(X,A)\bigr)
    \rightarrow
  \A_qC_p^G(X,A),
$$
defined by
$$\alpha_{p,q}([x\otimes gs])
  =\sign(g,s)\,\genh_{p+q}^G(\chi_{gs})
   \circ\zeta(G_s^g,gs)_{p+q}^{-1}(x)
$$
for all $x\in\genh_q^G(G/G_s^g)$ and $s\in C_p(X,A)$, is an isomorphism.
\end{theorem}

\begin{proof}
Set $n\doteq p+q$.  By lemmas \ref{lem:orbitmodule} and \ref{lem:simp-chains}, and theorem \ref{thm:simp-id}, we have that the following sequence of compositions is an isomorphism:
\begin{multline*}
  \genh_q^G\bigl(C_p(X,A)\bigr)
  \leftarrow
  \oplus_{s\in{\cal R}}\genh_q^G(G/G_s)
  \xleftarrow{\oplus\zeta(G_s,s)_n}\\
  \oplus_{\sigma\in{\cal R}}
    \genh_n^G(G/G_s\times s,G/G_s\times\D{s})
  \xrightarrow{\Theta_\ast}
  \genh_n^G(X_p,X_{p-1}).
\end{multline*}
We claim that this is precisely $\alpha_{p,q}$.  Indeed, $[x\otimes gs]=[\genh_q^G(\mu_g)(x)\otimes s]$, so that $[x\otimes gs]$ maps to $\genh_n^G(\chi_s)\circ\zeta(G_s,s)_n^{-1}\circ\genh_q^G(\mu_g)(x)$ under the above composition.  However,
\begin{align*}
  &\genh_n^G(\chi_s)\circ\zeta(G_s,s)_n^{-1}
   \circ\genh_q^G(\mu_g)\\
  &\quad\quad=(\sign g^{-1})\,\genh_n^G(\chi_s)
      \circ\genh_n^G(\mu_g\times g^{-1})
        \circ\zeta(G_s^g,gs)_n^{-1}\\
  &\quad\quad=(\sign g^{-1})\,\genh_n^G(\chi_{gs})
      \circ\zeta(G_s^g,gs)_n^{-1}
\end{align*}
by theorem \ref{thm:id-nat-perm} and functoriality.  Moreover, $\sign g^{-1}=\sign(g,s)$.
\end{proof}

\begin{lemma}
\label{lem:simp-bound-elem}
Suppose $G/H\times s$ is an equivariant $p$--simplex.  If we identify $s$ with $\{eH\}\times s$, then the following diagram commutes:
$$\begin{CD}
  \A_qC_p^G(G/H\times s)
    @>{\A_q\D_p}>>
      \A_qC_{p-1}^G(G/H\times s)\\
  @A{\alpha_{p,q}}AA
    @A{\alpha_{p-1,q}}AA\\
  \genh_q^G\bigl(C_p(G/H\times s)\bigr)
    @>{\D_p}>>
      \genh_q^G\bigl(C_{p-1}(G/H\times s)\bigr),
\end{CD}$$
where for all $x\in\genh_q^G(G/H^g)$, we have $\D_p[x\otimes gs]=\sum_{k=0}^p(-1)^k[x\otimes gs_k]$.
\end{lemma}

\begin{proof}
It suffices to verify the formula for $\D_p$ in the case $g=e$.  We set $X\doteq G/H\times s$.  Since $C_p(X)=\{gs\mid gH\in G/H\}$ and $C_{p-1}(X)=\{gs_k\mid gH\in G/H,0\leq k\leq p\}$, linearity implies that we only need to show that ${\it proj}_k\circ\D_p[x\otimes s]=(-1)^k[x\otimes s_k]$, where ${\it proj}_k:\genh_q^G\bigl(C_p(X)\bigr)\rightarrow\genh_q^G[s_k]$ is the projection map.  Now, $X_p=X^p=G/H\times s$ and $X_{p-1}=X^{p-1}=G/H\times\D s$.  Set $n\doteq p+q$.  We claim that in the diagram
\begin{equation}\label{eq:d-special}
\begin{CD}
  \genh_q^G\bigl(C_p(G/H\times s)\bigr)
    @>{\alpha_{p,q}}>>
      \genh_n^G(X_p,X_{p-1})\\
  @V{\D_p}VV
    @V{{\cal A}_q\D_p}VV\\
  \genh_q^G\bigl(C_{p-1}(G/H\times s)\bigr)
    @>{\alpha_{p-1,q}}>>
      \genh_{n-1}^G(X_{p-1},X_{p-2})\\
  @V{{\it proj}_k}VV
    @VVV\\
  \genh_q^G[s_k]
    @>{\alpha_{p-1,q}}>>
      \genh_{n-1}^G(G/H\times\D s,G/H\times\hat{s}_k)\\
  @|
    @AAA\\
  \genh_q^G[s_k]
    @>{\alpha_{p-1,q}}>>
      \genh_{n-1}^G(G/H\times s_k,G/H\times\D s_k),
\end{CD}
\end{equation}
(i) the bottom two squares commute, and that (ii) the composition of vertical arrows on the right is equal to $[s:s_k]_n^H$.  The formula for $\D_p$ then follows:
\begin{align*}
  \alpha_{p-1,q}\circ{\it proj}_k\circ\D_p([x\otimes s])
  &=[s:s_k]_n^H\circ\alpha_{p,q}([x\otimes s])\\
  &=[s:s_k]_n^H\circ\zeta(H,s)_n^{-1}(x)\\
  &=(-1)^k\zeta(H,s_k)_{n-1}^{-1}(x)\\
  &=(-1)^k\alpha_{p-1,q}([x\otimes s_k]),
\end{align*}
by theorems \ref{thm:simp-perm-id} and \ref{thm:simplex-chains}, and the fact that $\chi_s$ and $\chi_{s_k}$ are both the identity in this case.

Let us justify claim (i).  For the middle square of diagram \eqref{eq:d-special}, theorem \ref{thm:simplex-chains} gives
\begin{align*}
  \genh_{n-1}^G(i_k)\circ\alpha_{p-1,q}([x\otimes s_l])
  &=\genh_{n-1}^G(i_k)
    \circ\genh_{n-1}^G(\chi_{s_l})
    \circ\zeta(H,s_l)_{n-1}^{-1}(x)\\
  &=\genh_{n-1}^G(i_k\circ\chi_{s_l})
    \circ\zeta(H,s_l)_{n-1}^{-1}(x),
\end{align*}
where $i_k:(X_{p-1},X_{p-2})\rightarrow(G/H\times\D s,G/H\times\hat{s}_k)$ is the inclusion.  If $l=k$, then this is just $\alpha_{p-1,q}([x\otimes s_k])$.  However, if $l\neq k$, then $s_l\subseteq\hat{s}_k$, so that we can factor $i_k\circ\chi_{s_l}$ as
$$(G/H\times s_l,G/H\times\D s_l)
  \rightarrow
  (G/H\times\hat{s}_k,G/H\times\hat{s}_k)
  \rightarrow
  (G/H\times\D s,G/H\times\hat{s}_k)
$$
($\chi_{s_l}$ is the inclusion map in this case); hence $\genh_\ast^G(i_k\circ\chi_{s_l})=0$.  The bottom square in diagram \eqref{eq:d-special} also commutes by theorem \ref{thm:simplex-chains}.

To justify claim (ii), we see that the composition of the top two vertical arrows on the right in diagram \eqref{eq:d-special} is equal to the boundary map
$$\genh_n^G(G/H\times s,G/H\times\D s)
  \xrightarrow{\D_n}
  \genh_{n-1}^G(G/H\times\D s,G/H\times\hat{s}_k),
$$
by definition of the boundary map of a triple.  The claim now follows from definition of the incidence map (lemma \ref{lem:simpincidence}).
\end{proof}

\begin{theorem}
\label{thm:simplex-boundaries}
If $(X,A)$ is a $G$--pair of simplicial $G$--complexes, the diagram
$$\begin{CD}
  \A_qC_p^G(X,A)
    @>{\A_q\D_p}>>
      \A_qC_{p-1}^G(X,A)\\
  @A{\alpha_{p,q}}AA
    @A{\alpha_{p-1,q}}AA\\
  \genh_q^G\bigl(C_p(X,A)\bigr)
    @>{\D_p}>>
      \genh_q^G\bigl(C_{p-1}(X,A)\bigr)
\end{CD}$$
commutes, where for all $x\in\genh_q^G(G/G_s^g)$ and $s\in C_p(X,A)$,
$$\D_p[x\otimes gs]
  =\sum_{k=0}^p(-1)^k
     \left[\genh_q^G(\kappa(G_{s_k}^g,G_s^g))(x)
           \otimes gs_k\right].
$$
\end{theorem}

\begin{proof}
Again, it suffice to assume $g=e$.  We make the abbreviations $H\doteq G_s$ and $n\doteq p+q$.  Note that $s^{p-1}=\D s$.  The following diagram commutes:
$$\begin{CD}
  \genh_n^G(X_p,X_{p-1})
    @>{\A_q\D_p}>>
      \genh_{n-1}^G(X_{p-1},X_{p-2})\\
  @A{\genh_n^G(\chi_s)}AA
    @AA{\genh_{n-1}^G(\chi_s)}A\\
  \genh_n^G(G/H\times s,G/H\times\D s)
    @>{\D_n}>>
      \genh_{n-1}^G(G/H\times\D s,G/H\times s^{p-2})\\
  @A{\alpha_{p,q}}A{\zeta(H,s)_n^{-1}}A
    @AA{\alpha_{p-1,q}}A\\
      \genh_q^G[s]
    @>{\D_p}>>
      \oplus_{k=0}^p\genh_q^G[s_k].
\end{CD}$$
Indeed, the top square commutes by proposition \ref{prop:triple}, and the bottom square commutes by lemma \ref{lem:simp-bound-elem}.  Thus we have
\begin{align*}
  \A_q\D_p\circ\alpha_{p,q}([x\otimes s])
  &=\A_q\D_p\circ\genh_n^G(\chi_s)\circ\zeta(H,s)_n^{-1}(x)\\
  &=\genh_{n-1}^G(\chi_s)\circ\alpha_{p-1,q}
    \circ\D_p[x\otimes s]\\
  &=\sum_{k=0}^p(-1)^k\genh_{n-1}^G(\chi_s)\circ
    \alpha_{p-1,q}([x\otimes s_k])\\
  &=\sum_{k=0}^p(-1)^k\genh_{n-1}^G(\chi_s)\circ
    \zeta(H,s_k)_{n-1}^{-1}(x),
\end{align*}
since $\chi_{s_k}:(G/H\times s_k,G/H\times\D s_k)\rightarrow(G/H\times\D s,G/H\times s^{p-2})$ is the inclusion.  Set $H_k\doteq G_{s_k}$ and $\kappa_k\doteq\kappa(H_k,H)$, and let $i:(s,\D s)\rightarrow(X_p,X_{p-1})$ denote the inclusion map.  Since $H\leq H_k$, one sees that $\chi_{s_k}\circ(\kappa_k\times i)=\chi_s$ when restricted to $(G/H\times s_k,G/H\times\D s_k)$.  Thus, by lemma \ref{lem:id-nat},
\begin{align*}
  \A_q\D_p\circ\alpha_{p,q}([x\otimes\sigma])
  &=\sum_{k=0}^p(-1)^k\genh_{n-1}^G(\chi_{s_k})\circ
    \genh_{n-1}^G(\kappa_k\times i)\circ
    \zeta(H,s_k)_{n-1}^{-1}(x)\\
  &=\sum_{k=0}^p(-1)^k\genh_{n-1}^G(\chi_{s_k})\circ
    \zeta(H_k,s_k)_{n-1}^{-1}\circ\genh_q^G(\kappa_k)(x)\\
  &=\sum_{k=0}^p(-1)^k\alpha_{p-1,q}
    ([\genh_q^G(\kappa_k)(x)\otimes s_k]).\qedhere
\end{align*}
\end{proof}

\begin{theorem}
\label{thm:Aq}
If $(X,A)$ is a $G$--pair of simplicial $G$--complexes, then we have $\A_q\genh_p^G(X,A)\cong\A_0H_p^G(X,A;\genh_q^G)$.
\end{theorem}

\begin{proof}
By proposition \ref{prop:singhom}, $H_0^G$ (with coefficients in $\genh_q^G$ understood) and $\genh_q^G$ define equivalent covariant coefficient systems upon restriction to ${\cal O}(G)$.  In particular, $H_0^G\bigl(C_p(X,A);\genh_q^G\bigr)=\genh_q^G\bigl(C_p(X,A)\bigr)$.  That is, the associated chains for $(\genh_\ast^G,\D_\ast)$ and $(H_\ast^G,\D_\ast)$ are isomorphic by theorem \ref{thm:simplex-chains}.  Furthermore by theorem \ref{thm:simplex-boundaries}, under this isomorphism the boundary maps are identical since $[\genh_q^G(\kappa(G_{s_k},G_s))(x)\otimes s_k]=[H_0^G(\kappa(G_{s_k},G_s);\genh_q^G)(x)\otimes s_k]$.
\end{proof}

\subsection{Consequences of the spectral sequence}

As in the previous subsection, we assume that we are given a $G$--pair $(X,A)$ of simplicial $G$--complexes, and that the generalized equivariant homology theory $(\genh_\ast^G,\D_\ast)$ satisfies the arbitrary sum property, or that $X$ (hence $A$) is finite in each dimension.  In addition, to make use of the spectral sequence, we must also assume that $(X,A)$ is homologically stable; that is, the $G$--filtration $\{X_p\doteq X^p\cup A\}_{p\in\Z}$ is homologically stable.  This is necessarily the case if $(X,A)$ is finite dimensional; in the next subsection, we will give criteria under which this is the case for the infinite dimensional case.

\begin{theorem}
\label{thm:sing-chains}
If $\M$ is a covariant coefficient system and $(X,A)$ is a $G$--pair of simplicial $G$--complexes, then
$$H_n^G(X_p,X_{p-1};\M)
  \cong
  \begin{cases}
    \M(C_p(X,A))          &\text{if $n=p$,}\\
    0                     &\text{if $n\neq p$}.
  \end{cases}
$$
\end{theorem}

\begin{proof}
This follows from theorem \ref{thm:simplex-chains} and proposition \ref{prop:singhom}.
\end{proof}

\begin{theorem}
\label{thm:simp-singhom}
Suppose $(X,A)$ is a $G$--pair of simplicial $G$--complexes that is stable with respect to equivariant singular homology, and that $\M$ is a coefficient system. Then $H_n^G(X,A;\M)\cong\A_0H_n^G(X,A;\M)$.
\end{theorem}

\begin{proof}
By theorem \ref{thm:genspec}, the spectral sequence \eqref{eq:genDEdef} converges with $E_{p,q}^2=\A_qH_p^G(X,A;\M)$.  However by theorem \ref{thm:sing-chains}, $E_{p,q}^1=0$ if $q\neq 0$.  So by dimensional considerations, the spectral sequence collapses at the $E^2$--term.  Moreover by theorem \ref{thm:sing-chains} and the exact sequence of the triple $(X_r,X_{r-1},A)$, we have (i) $H_p^G(X_r,A;\M)\cong H_p^G(X_{r-1},A;\M)$ for all $r\neq p,p+1$, and (ii) $H_p^G(X_p,A;\M)\rightarrow H_p^G(X_{p+1},A;\M)$ is an epimorphism.  By induction, (i) yields $H_p^G(X_{p-1},A;\M)\cong H_p^G(A,A;\M)=0$.  Therefore, in the notation of section \ref{sec:spectral}, we have $\phi^{p-1}H_p^G(X,A;\M)=0$.  From (i) we also have $H_p^G(X_{p+1},A;\M)\cong H_p^G(X_{p+k},A;\M)$ for all $k>1$; and so
$$\phi^pH_p^G(X,A;\M)
  ={\rm im}\{H_p^G(X_p,A;\M)\rightarrow H_p^G(X,A;\M)\}
  =H_p^G(X,A;\M),
$$
by (ii) and homological stability.  Consequently, $E_{p,0}^2\cong\phi^pH_p^G(X,A;\M)\cong H_p^G(X,A;\M)$.
\end{proof}

\begin{theorem}
If $(X,A)$ is a $G$--pair simplicial $G$--complexes that is stable with respect to both equivariant singular homology and $(\genh_\ast^G,\D_\ast)$, then the spectral sequence of the filtration $\{X^n\cup A\mid n\in\Z\}$ converges to $\genh_\ast^G(X,A)$, with $E_{p,q}^2\cong H_p^G(X,A;\genh_q^G)$.
\end{theorem}

\begin{proof}
The spectral sequence converges with $E_{p,q}^2=\A_q\genh_p^G(X,A)$, by theorem \ref{thm:genspec}.  By theorems \ref{thm:simp-singhom} and \ref{thm:Aq}, $E_{p,q}^2\cong\A_0H_p^G(X,A;\genh_q^G)\cong H_p^G(X,A;\genh_q^G)$.
\end{proof}

\subsection{A criterion for homological stability}

\begin{theorem}\label{thm:h-stable}
Suppose that $(\genh_\ast^G,\D_\ast)$ has compact supports, and that there exists an integer $n_0$ such that $\genh_n^G(G/H)=0$ for all $n\leq n_0$ and all $H\leq G$.  If $(X,A)$ is a $G$--pair of simplicial $G$--complex such that the subspaces $X_p$ are compact, then $\{X_p\}_{p\in\Z}$ is homologically stable with respect to $(\genh_\ast^G,\D_\ast)$.
\end{theorem}

\begin{proof}
Set $p_0\doteq n-n_0$.  Theorem \ref{thm:simplex-chains} implies that $\genh_{n+1}^G(X_{p+1},X_p)$ is isomorphic to $\genh_{n-p}^G(C_{p+1}(X_{p+1},X_p))$; and thus $\genh_{n+1}^G(X_{p+1},X_p)=0$ for $n-p\leq n_0$.  From the exact sequence
$$\genh_{n+1}^G(X_{p+1},X_p)
  \xrightarrow{\D_n}\genh_n^G(X_p,A)
  \rightarrow\genh_n^G(X_{p+1},A)
  \rightarrow\genh_n^G(X_{p+1},X_p),
$$
$\genh_n^G(X_p,A)\rightarrow\genh_n^G(X_{p+1},A)$ is an isomorphism for all $p\geq p_0$.  It follows that $\genh_n^G(X_{p_0},A)\rightarrow\genh_n^G(X_p,A)$ is an isomorphism for all $p\geq p_0$; and hence inclusion induces $\genh_n^G(X_p,A)\cong\lim_{r\geq p}\genh_n^G(X_r,A)$, for each $p\geq p_0$ (see corollary 4.8 of chapter VIII in \cite{ES}).  Now, $\{(X_r,A)\mid r\geq p\}$ is cofinal in the set ${\cal C}$ of all compact pairs $(Y,B)\subseteq(X,A)$; hence inclusion also induces $\lim_{r\geq p}\genh_n^G(X_r,A)\cong\lim_{(Y,B)\in{\cal C}}\genh_n^G(Y,B)$ (see corollary 4.14 of chapter VIII in \cite{ES}).  The lemma now follows by proposition \ref{prop:limit}.
\end{proof}

The subspaces $X_p$ will be compact, for instance, when $X$ is Hausdorff and finite in each dimension, and $G$ is compact.  In particular, theorem \ref{thm:h-stable} applies to equivariant singular homology in this case.


\end{document}